\documentclass{amsart}
\usepackage{amsfonts}
\usepackage{amsmath}
\usepackage{amssymb}
\usepackage{amsthm}
\newtheorem{theorem}{Theorem}
\newtheorem{prop}{Proposition}
\newtheorem{lemma}{Lemma}
\newtheorem{rem}{Remark}
\newtheorem{exmp}{Example}
\newtheorem{cor}{Corollary}
\newtheorem{fact}{Fact}

\begin{document}
\author{Mark Pankov}
\title{Isometric embeddings of Johnson graphs in Grassmann graphs}
\address{Department of Mathematics and Informatics, University of Warmia and Mazury,
{\. Z}olnierska 14A, 10-561 Olsztyn, Poland}
\email{pankov@matman.uwm.edu.pl}

\maketitle

\begin{abstract}
Let $V$ be an $n$-dimensional vector space ($4\le n <\infty$) and
let ${\mathcal G}_{k}(V)$ be the Grassmannian formed by all $k$-dimensional subspaces of $V$.
The corresponding Grassmann graph will be denoted by $\Gamma_{k}(V)$.
We describe all isometric embeddings of Johnson graphs $J(l,m)$, $1<m<l-1$
in $\Gamma_{k}(V)$, $1<k<n-1$ (Theorem 4).
As a consequence, we get the following:
the image of every isometric embedding
of $J(n,k)$ in $\Gamma_{k}(V)$ is an apartment of ${\mathcal G}_{k}(V)$
if and only if $n=2k$.
Our second result (Theorem 5) is a classification of rigid isometric embeddings of Johnson graphs
in $\Gamma_{k}(V)$, $1<k<n-1$.
\end{abstract}

\section{Introduction}
A {\it building} \cite{Tits} is a simplicial complex $\Delta$
together with a family of subcomplexes called {\it apartments}
and satisfying a certain collection of axioms.
Maximal simplices of $\Delta$ are called {\it chambers}.
They have the same cardinal number $n$ (the rank of $\Delta$)
and we say that two chambers are {\it adjacent} if their intersection consists of $n-1$ vertices.
Denote by ${\rm Ch}(\Delta)$ the set of all chambers.
Consider the graph $\Gamma_{\rm ch}(\Delta)$
whose vertex set is ${\rm Ch}(\Delta)$ and whose edges are pairs of adjacent chambers.
Let ${\mathcal A}$ be the intersection of ${\rm Ch}(\Delta)$ with an apartment
and let $\Gamma({\mathcal A})$ be the restriction of the graph $\Gamma_{\rm ch}(\Delta)$ to ${\mathcal A}$.
There is the following characterization of the intersections of ${\rm Ch}(\Delta)$ with apartments 
(see, for example, \cite{Brown} p. 90).

\begin{theorem}
A subset of ${\rm Ch}(\Delta)$ is the intersection of ${\rm Ch}(\Delta)$
with an apartment if and only if it is
the image of an isometric embedding of $\Gamma({\mathcal A})$ in $\Gamma_{\rm ch}(\Delta)$.
\end{theorem}

The vertex set of a building can be labeled by the nodes of the associated diagram
(the labeling is unique up to a permutation on the set of nodes).
The set of vertices corresponding to the same node is called a {\it Grassmannian}
(more general Grassmannians defined by parts of the diagram were considered in \cite{Pasini}).

Let $\Delta$ be a building and let ${\mathcal G}$ be one of its Grassmannians.
We say that two distinct elements $a,b\in {\mathcal G}$ are {\it adjacent} if
there exists a simplex $P\in \Delta$ such that $P\cup\{a\}$ and $P\cup\{b\}$ both are
chambers;
in this case, the set of all $c\in {\mathcal G}$ such that $P\cup\{c\}$ is a chamber
is said to be the {\it line} joining $a$ and $b$.
The Grassmannian  ${\mathcal G}$ together with the set of all such lines is a partial linear space
(each point belongs to a line, each line contains at least two points,
for any two distinct points there is at most one line containing them).
This partial linear space will be called the {\it Grassmann space} corresponding to ${\mathcal G}$.
The associated collinearity graph is the {\it Grassmann graph} $\Gamma_{\mathcal G}$
whose vertex set is ${\mathcal G}$ and whose edges are pairs of adjacent elements.
The intersections of ${\mathcal G}$ with apartments of $\Delta$ are said to be
{\it apartments} of the Grassmannian ${\mathcal G}$.

In \cite{CS,CKS,Kasikova1,Kasikova2} apartments of some Grassmannians
were characterized in terms of the associated Grassmann spaces.
It is natural to ask whether there are metric characterizations of apartments of Grassmannians similar to Theorem 1?

Every building of type $\textsf{A}_{n-1}$ ($n\ge 4$) is
the flag complex of a certain  $n$-dimensional vector space $V$.
The Grassmannians of this building are the usual Grassmannians
${\mathcal G}_{k}(V)$, $k\in\{1,\dots,n-1\}$, where ${\mathcal G}_{k}(V)$
is formed by all $k$-dimensional subspaces of $V$.
The corresponding Grassmann graph is denoted by $\Gamma_{k}(V)$.
The restriction of $\Gamma_{k}(V)$ to every apartment of ${\mathcal G}_{k}(V)$
is isomorphic to the Johnson graph $J(n,k)$.
Recall that the vertex set of $J(n,k)$ is formed by all $k$-element subsets of an $n$-element set;
two such subsets are adjacent (joined by an edge) if their intersection consists of $k-1$ elements.
%It was mentioned in \cite{CS} that
%$\Gamma_{k}(V)$ contains subgraphs isomorphic to $J(n,k)$
%whose vertex sets are not apartments of ${\mathcal G}_{k}(V)$.

Every apartment of ${\mathcal G}_{k}(V)$ is the image of an isometric embedding of $J(n,k)$ in $\Gamma_{k}(V)$.
We show that
the image of every isometric embedding of $J(n,k)$ in $\Gamma_{k}(V)$ is an apartment of ${\mathcal G}_{k}(V)$
if and only if $n=2k$.
This statement follows from our classification
of isometric embeddings of Johnson graphs $J(l,m)$, $1<m<l-1$ in the Grassmann graph $\Gamma_{k}(V)$,
$1<k<n-1$ (Theorem 4).

If $2k\le n$ then any finite $(2k)$-independent subset ${\mathcal X}\subset {\mathcal G}_{1}(V)$
(every $(2k)$-element subset of ${\mathcal X}$ is independent)
defines an isometric embedding of $J(n',k)$, $n'=|{\mathcal X}|$ in $\Gamma_{k}(V)$.
The image consists of all $S\in {\mathcal G}_{k}(V)$ such that $S$ is the sum of $k$ elements from
${\mathcal X}$.
All other isometric embeddings of Johnson graphs in $\Gamma_{k}(V)$ are some modifications of this construction.

Our second result (Theorem 5) is a classification of rigid isometric embeddings of Johnson graphs in
$\Gamma_{k}(V)$.
The term {\it rigid} means that
every automorphism of the restriction of $\Gamma_{k}(V)$ to the image can be extended to an automorphism of $\Gamma_{k}(V)$.
In the case when $n\ne 2k$, apartments of ${\mathcal G}_{k}(V)$ satisfy this condition;
but there exist other rigid isometric embeddings of $J(n,k)$ in $\Gamma_{k}(V)$.

\section{Grassmannians of type $\textsf{A}_{n-1}$}

Let $V$ be an $n$-dimensional left vector space over a division ring $R$ and let $4\le n<\infty$.
Two elements of ${\mathcal G}_{k}(V)$ are adjacent if their intersection
is $(k-1)$-dimensional; this is equivalent to the fact that
their sum is $(k+1)$-dimensional.

Let $M$ and $N$ be subspaces of $V$ such that  $M\subset N$ and
$$\dim M <k<\dim N.$$
Denote by $[M,N]_{k}$ the set of all $S\in {\mathcal G}_{k}(V)$
satisfying $M\subset S\subset N$.
If $M=0$ or $N=V$ then this set will be denoted by
$\langle N]_{k}$ or $[M\rangle_{k}$, respectively.
If
$$\dim M=k-1\;\mbox{ and }\;\dim N=k+1$$
then $[M,N]_{k}$ is a line of the Grassmann space corresponding to ${\mathcal G}_{k}(V)$;
we denote this Grassmann space by ${\mathfrak G}_{k}(V)$.
In the case when $k=1,n-1$, this is the projective space $\Pi_{V}$ associated with $V$
or the dual projective space $\Pi^{*}_{V}$, respectively.

If $k=1,n-1$ then any two distinct vertices of the Grassmann graph $\Gamma_{k}(V)$ are adjacent.
In the case when $1<k<n-1$, there are precisely the following two types of maximal cliques of $\Gamma_{k}(V)$
(see, for example, Section 3.1 \cite{Pankov}):
\begin{enumerate}
\item[$\bullet$] the stars $[M\rangle_{k}$, $M\in {\mathcal G}_{k-1}(V)$;
\item[$\bullet$] the tops  $\langle N]_{k}$, $N\in {\mathcal G}_{k+1}(V)$.
\end{enumerate}
The set of all maximal cliques of $\Gamma_{k}(V)$
coincides with the set of all maximal singular subspaces of ${\mathfrak G}_{k}(V)$
(a subspace of a partial linear space is called singular if any two distinct points of this subspace
are joined by a line), see Section 3.1 \cite{Pankov}.

Denote by $d_{k}$ the distance in $\Gamma_{k}(V)$.
For all $S,U\in {\mathcal G}_{k}(V)$ we have
$$d_{k}(S,U)=k-\dim(S\cap U)=\dim(S+U)-k.$$
If ${\mathcal X}\subset {\mathcal G}_{k}(V)$
then we write $\Gamma({\mathcal X})$ for the restriction of the Grassmann graph $\Gamma_{k}(V)$
to ${\mathcal X}$.

Every apartment of ${\mathcal G}_{k}(V)$ is defined by a base of $V$;
it consists of all elements of ${\mathcal G}_{k}(V)$ spanned by subsets of this base.
Let $B$ be a base of $V$ and ${\mathcal A}_{k}\subset {\mathcal G}_{k}(V)$
be the associated apartment.
Clearly, ${\mathcal A}_{1}$ and ${\mathcal A}_{n-1}$ are bases of the projective spaces $\Pi_{V}$
and $\Pi^{*}_{V}$, respectively.
Suppose that $1<k<n-1$.
It was noted above that $\Gamma({\mathcal A}_{k})$ is isomorphic to $J(n,k)$.
The maximal cliques of $\Gamma({\mathcal A}_{k})$ are of the following two
types:
\begin{enumerate}
\item[$\bullet$] the stars ${\mathcal A}_{k}\cap [M\rangle_{k}$, $M\in {\mathcal A}_{k-1}$,
\item[$\bullet$] the tops  ${\mathcal A}_{k}\cap \langle N]_{k}$, $N\in {\mathcal A}_{k+1}$.
\end{enumerate}
Every maximal clique of $\Gamma({\mathcal A}_{k})$ is an independent
subset of ${\mathfrak G}_{k}(V)$ spanning a maximal singular
subspace (a subset $X$ of a partial linear space is independent if
the subspace spanned by $X$ is not spanned by a proper subset of $X$).

For every subspace $S\subset V$ the {\it annihilator} $S^{0}$ is
the subspace formed by all $x^{*}\in V^{*}$ that vanish on $S$.
If ${\mathcal X}\subset {\mathcal G}_{k}(V)$ then we write ${\mathcal X}^{0}$
for the set formed by the annihilators of all $S\in{\mathcal X}$.
The {\it annihilator mapping} of ${\mathcal G}_{k}(V)$ to ${\mathcal G}_{n-k}(V^{*})$
is the bijection transferring every $S\in {\mathcal G}_{k}(V)$ to the annihilator $S^{0}$.
The following statement is well-known.

\begin{fact}\label{fact1}
The annihilator mapping of
${\mathcal G}_{k}(V)$ to ${\mathcal G}_{n-k}(V^{*})$
is an isomorphism of $\Gamma_{k}(V)$ to $\Gamma_{n-k}(V^{*})$
sending apartments to apartments, stars to tops and tops to stars.
\end{fact}

Let $M$ and $N$  be subspaces of $V$ such that $M\subset N$ and
$$\dim M=m <k<l=\dim N.$$
Then $[M,N]_{k}$ is a subspace of ${\mathfrak G}_{k}(V)$;
subspaces of such type are called {\it parabolic} \cite{CKS}.
Let $B$ be a base of $V$ such that $M$ and $N$ are spanned by subsets of $B$.
The intersection of the associated apartment of ${\mathcal G}_{k}(V)$ with
$[M,N]_{k}$ is said to be an {\it apartment of the parabolic subspace} $[M,N]_{k}$.
The natural collineation (isomorphism) of $[M,N]_{k}$ to the Grassmann space ${\mathfrak G}_{k-m}(N/M)$
establishes a one-to-one correspondence between apartments.
The restrictions of $\Gamma_{k}(V)$ to apartments of $[M,N]_{k}$
are isomorphic to $J(l-m,k-m)$.
It is clear that the annihilator mapping ${\mathcal G}_{k}(V)$ to ${\mathcal G}_{n-k}(V^{*})$
transfers $[M,N]_{k}$ to $[N^{0}, M^{0}]_{n-k}$; moreover, it establishes a one-to-one
correspondence between apartments of these parabolic subspaces.

It was noted above that maximal cliques of
the restriction of $\Gamma_{k}(V)$ to an apartment of ${\mathcal G}_{k}(V)$
are independent subsets of ${\mathfrak G}_{k}(V)$ (the same holds for apartments of
parabolic subspaces). This property characterizes apartments.

\begin{theorem}[B. N. Cooperstein, A. Kasikova, E.E. Shult \cite{CKS}]
Let $1<k<n-1$ and $l,m$ be natural numbers satisfying $k<l\le n$ and $0\le m<k$.
Let also ${\mathcal X}$ be a subset of ${\mathcal G}_{k}(V)$
such that $\Gamma({\mathcal X})$ is isomorphic to the Johnson graph $J(l-m,k-m)$ and
every maximal clique of $\Gamma({\mathcal X})$ is an independent subset of
${\mathfrak G}_{k}(V)$.
Then ${\mathcal X}$ is an apartment in a parabolic subspace of ${\mathfrak G}_{k}(V)$;
in the case when $k=m$ and $n=l$, this is an apartment of ${\mathcal G}_{k}(V)$.
\end{theorem}

\begin{rem}\rm{
Also, in \cite{CKS} such kind of characterizations
were established for apartments of dual polar spaces and half-spin Grassmannians.
More general results can be found in \cite{Kasikova1,Kasikova2}.
}\end{rem}

Let $V'$ be a left vector space over a division ring $R'$.
A mapping $u:V\to V'$ is called {\it semi-linear}
if it is additive, i.e.
$$u(x+y)=u(x)+u(y)$$
for all $x,y\in V$,
and there exists a homomorphism $\sigma:R\to R'$ such that
$$u(ax)=\sigma(a)u(x)$$ for all $x\in V$ and $a\in R$.
A semi-linear bijection $u:V\to V'$ is called a {\it semi-linear isomorphism} if
the associated homomorphism of $R$ to $R'$ is an isomorphism.
Every semi-linear automorphism of $V$ induces an automorphism of the flag complex $\Delta(V)$
whose restriction to each ${\mathcal G}_{k}(V)$ is an automorphism of $\Gamma_{k}(V)$.

Let $u$ be a semi-linear isomorphism of $V$ to $V^{*}$.
The mapping transferring every subspace $S\subset V$ to the annihilator of $u(S)$
is an automorphism of $\Delta(V)$.
The restriction of this automorphism to ${\mathcal G}_{k}(V)$ is an isomorphism of
$\Gamma_{k}(V)$ to $\Gamma_{n-k}(V)$.
In the case when $n=2k$, this is an automorphism of $\Gamma_{k}(V)$.

Recall that $V^{*}$ is a left vector space over the opposite division ring $R^{*}$
(the division rings $R$ and $R^{*}$ have the same set of elements and the same additive operation,
the multiplicative operation $*$ on $R^{*}$ is defined by $a*b:=b\cdot a$ if $\cdot$ is the multiplicative operation on $R$).
Semi-linear isomorphisms of $V$ to $V^{*}$ exist only in the case when
$R$ and $R^{*}$ are isomorphic.
Note that $R=R^{*}$ if $R$ is commutative.

By the Fundamental Theorem of Projective Geometry \cite{Baer}, every automorphism of the complex $\Delta(V)$
is induced by a semi-linear automorphism of $V$ or a semi-linear isomorphism of $V$ to $V^{*}$.
Two semi-linear mappings define the same automorphism of $\Delta(V)$ if and only if one of them
is a scalar multiple of the other.
Denote by ${\rm PGL}(V)$ the group of all automorphisms of $\Delta(V)$  induced by {\it linear} automorphisms of $V$.
This group coincides with the group of all automorphisms of $\Delta(V)$ induced by semi-linear automorphisms of $V$
only in the case when every automorphism of $R$ is inner
(for example, if $R={\mathbb R}$ or $R$ is the division ring of real quaternion numbers).

%Classical Chow's theorem can be formulated in the following form.

\begin{theorem}[W.L. Chow \cite{Chow}]
Every automorphism of $\Gamma_{k}(V)$, $1<k<n-1$
has a unique extension to an automorphism of $\Delta(V)$,
in other words, it is induced by a semi-linear automorphism of $V$
or a semi-linear isomorphism of $V$ to $V^{*}$. The second possibility can be realized
only in the case when $n=2k$.
\end{theorem}

\begin{rem}{\rm
Some generalizations of Chow's theorem can be found in \cite{Pankov}.
}\end{rem}

Let $u:V\to V'$ be a semi-linear isomorphism.
The {\it contragradient} $\check{u}$ is the semi-linear isomorphism
$$(u^{*})^{-1}:V^{*}\to V'^{*}$$
(the inverse of the adjoint mapping), see Section 1.3.3 \cite{Pankov}.
It transfers the annihilator of every subspace $S\subset V$
to the annihilator of $u(S)$.
The contragradient $\check{u}$ is linear if and only if $u$ is linear.
If $u$ is a semi-linear automorphism of $V$
then the mapping transferring every subspace $S\subset V^{*}$ to $u(S^{0})^{0}$ is the automorphism of
the flag complex $\Delta(V^{*})$ induced by ${\check u}$.

\section{Isometric embeddings of Johnson graphs in $\Gamma_{k}(V)$}
In what follows the images of isometric embeddings of Johnson graphs in $\Gamma_{k}(V)$
will be called $J$-{\it subsets} of ${\mathcal G}_{k}(V)$.

In this section we give several examples of $J$-subsets and state our first result (Theorem 4) ---
a classification of isometric embeddings of $J(l,m)$, $1<m<l-1$ in
$\Gamma_{k}(V)$, $1<k<n-1$.
The existence of such embeddings
means  that $\min\{m,l-m\}$ (the diameter of $J(l,m)$) is not greater
than $\min\{k,n-k\}$ (the diameter of $\Gamma_{k}(V)$).

A subset ${\mathcal X}\subset {\mathcal G}_{1}(V)$ is called $m$-{\it independent}
if every $m$-element subset of ${\mathcal X}$ is independent in the projective space $\Pi_{V}$,
in other words, the sum of any $m$ elements from ${\mathcal X}$ belongs to ${\mathcal G}_{m}(V)$.
An $n$-independent subset of ${\mathcal G}_{1}(V)$ consisting of $n$ elements is a base of $\Pi_{V}$.

An $m$-independent subset of ${\mathcal G}_{1}(V)$ consisting of $m+1$ elements is called an $m$-{\it simplex}
of $\Pi_{V}$ if it is not independent.
If $P_{1},\dots,P_{m+1}\in {\mathcal G}_{1}(V)$ form  an $m$-simplex then there exist linearly
independent vectors $x_{1},\dots,x_{m}\in V$ such that
$$P_{1}=\langle x_{1}\rangle,\dots,P_{m}=\langle x_{m}\rangle\;\mbox{ and }\;
P_{m+1}=\langle x_{1}+\dots+x_{m}\rangle.$$

\begin{prop}
If the division ring $R$ is infinite
then for every natural number $n'>n$ there exists an $n$-independent subset of $\Pi_{V}$ consisting of $n'$ elements.
\end{prop}

\begin{proof}
The statement can be proved induction by $n'$.
Suppose that ${\mathcal X}\subset {\mathcal G}_{1}(V)$ is an $n$-independent subset consisting of $n'-1$ elements.
Denote by ${\mathcal Y}$ be the set of all $S\in{\mathcal G}_{n-1}(V)$
such that $S$ is the sum of $n-1$ elements from ${\mathcal X}$. Since $R$ is infinite,
there exists $P\in {\mathcal G}_{1}(V)$ satisfying $P\not\subset S$ for every $S\in {\mathcal Y}$.
The subset ${\mathcal X}\cup \{P\}$ is $n$-independent.
\end{proof}

Now we generalize Example 1 from \cite{CS} (Section 9).

\begin{exmp}\label{exmp2}{\rm
Suppose that $2k\le n$ and ${\mathcal X}=\{P_{1},\dots,P_{n'}\}$ is a $2k$-independent subset of ${\mathcal G}_{1}(V)$.
We denote
${\mathcal J}_{k}({\mathcal X})$ the set of all $k$-dimensional subspaces
of type $P_{i_{1}}+\dots+P_{i_{k}}$.
The mapping
$$\{i_{1},\dots,i_{k}\}\to P_{i_{1}}+\dots+P_{i_{k}}$$
is an isometric embedding of $J(n',k)$ in $\Gamma_{k}(V)$;
indeed, it is easy to see that
$$\dim[(P_{i_{1}}+\dots+P_{i_{k}})\cap (P_{j_{1}}+\dots+P_{j_{k}})]=
|\{i_{1},\dots,i_{k}\}\cap \{j_{1},\dots,j_{k}\}|.$$
If $n'=n$ and ${\mathcal X}$ is a base of $\Pi_{V}$ then  ${\mathcal J}_{k}({\mathcal X})$ is an apartment of ${\mathcal G}_{k}(V)$.
If $n'\le n+1$ and ${\mathcal X}$ is an $(n'-1)$-simplex then ${\mathcal J}_{k}({\mathcal X})$ is the set of its $(k-1)$-faces.
}\end{exmp}

A subset ${\mathcal Y}\subset {\mathcal G}_{n-1}(V)$ is called $m$-{\it independent}
if every $m$-element subset of ${\mathcal X}$ is independent in the projective space $\Pi^{*}_{V}$,
this means that the intersection of any $m$ elements from ${\mathcal X}$ belongs to ${\mathcal G}_{n-m}(V)$.
An $n$-independent subset of ${\mathcal G}_{n-1}(V)$ consisting of $n$ elements is a base of $\Pi^{*}_{V}$.
An $m$-independent subset of ${\mathcal G}_{n-1}(V)$ consisting of $m+1$ elements is said to be
an $m$-{\it simplex} of $\Pi^{*}_{V}$ if it is not independent.
The annihilator mapping of ${\mathcal G}_{i}(V)$, $i=1,n-1$ to ${\mathcal G}_{n-i}(V^{*})$ transfers
$m$-independent subsets to  $m$-independent subsets, in particular, simplices go to simplices.

\begin{exmp}\label{exmp3}{\rm
Suppose that $2k\ge n$ and ${\mathcal Y}=\{S_{1},\dots,S_{n'}\}$ is a $(2n-2k)$-independent subset of ${\mathcal G}_{n-1}(V)$.
The intersection of any $n-k$ elements from ${\mathcal Y}$ belongs to ${\mathcal G}_{k}(V)$
and we denote by ${\mathcal J}^{*}_{k}({\mathcal Y})$ the set of all such intersections.
This is a $J$-subset.
Indeed, ${\mathcal Y}^{0}$ is a $(2n-2k)$-independent subset of ${\mathcal G}_{1}(V^{*})$ and
the annihilator mapping of ${\mathcal G}_{n-k}(V^{*})$ to ${\mathcal G}_{k}(V)$
transfers ${\mathcal J}_{n-k}({\mathcal Y}^{0})$ (see Example \ref{exmp2})
to ${\mathcal J}^{*}_{k}({\mathcal Y})$.
If $n'=n$ and ${\mathcal Y}$ is a base of $\Pi^{*}_{V}$ then  ${\mathcal J}^{*}_{k}({\mathcal Y})$ is an apartment of ${\mathcal G}_{k}(V)$.
If $n'\le n+1$ and ${\mathcal Y}$ is an $(n'-1)$-simplex then ${\mathcal J}^{*}_{k}({\mathcal Y})$ is the set of its $(n-k-1)$-faces.
}\end{exmp}

We will need the following modifications of
Examples \ref{exmp2} and \ref{exmp3}.

\begin{exmp}\label{exmp4}{\rm
Let $M$ be a $(k-m)$-dimensional subspace of $V$ such that $m>1$ and $m+k\le n$.
Then
$$2m\le n-k+m.$$
Consider the $(n-k+m-1)$-dimensional projective space $[M\rangle_{k-m+1}$
(it can be identified with $\Pi_{V/M}$).
Let ${\mathcal X}$ be a finite $(2m)$-independent subset of $[M\rangle_{k-m+1}$.
By Example \ref{exmp2},
${\mathcal J}_{m}({\mathcal X})$ is a $J$-subset contained in
the parabolic subspace $[M\rangle_{k}$;
it is the image of an isometric embedding of $J(l,m)$, $l=|{\mathcal X}|$ in $\Gamma_{k}(V)$.
In the case when $M=0$, we get the $J$-subset constructed in Example \ref{exmp2}.
Let $N$ be the sum of all elements from ${\mathcal X}$ (possible $N=V$).
If ${\mathcal X}$ is a base of the projective space $[M,N]_{k-m+1}$ then
${\mathcal J}_{m}({\mathcal X})$ is an apartment of the parabolic subspace $[M,N]_{k}$.
If $$l\le n-k+m+1$$ and ${\mathcal X}$ is an $(l-1)$-simplex of $[M\rangle_{k-m+1}$ then
${\mathcal J}_{m}({\mathcal X})$ is the set of  its $(m-1)$-faces.
}\end{exmp}

\begin{exmp}\label{exmp5}{\rm
Let $N$ be a $(k+m)$-dimensional subspace of $V$ such that $1<m\le k$.
Then $2m\le k+m$.
Consider the $(k+m-1)$-dimensional projective space $\langle N]_{k+m-1}$.
Let ${\mathcal Y}$ be a finite $(2m)$-independent subset of $\langle N]_{k+m-1}$.
As in Example \ref{exmp3}, the intersection of any $m$ elements from ${\mathcal Y}$ is $k$-dimensional
and we define ${\mathcal J}^{*}_{k}({\mathcal Y})$.
This is the image of an isometric embedding of $J(l,m)$, $l=|{\mathcal Y}|$
in $\Gamma_{k}(V)$.
If $N=V$ then we get the $J$-subset constructed in Example \ref{exmp3}.
Let $M$ be the intersection of all elements from ${\mathcal Y}$
(possible $M=0$). If ${\mathcal Y}$ is a base of the projective space $[M,N]_{k+m-1}$
then ${\mathcal J}^{*}_{k}({\mathcal Y})$  is an apartment of the parabolic subspace $[M,N]_{k}$.
If $$l\le k+m+1$$ and ${\mathcal Y}$ is an $(l-1)$-simplex of $\langle N]_{k+m-1}$ then
${\mathcal J}^{*}_{k}({\mathcal Y})$ is the set of its $(m-1)$-faces.
}\end{exmp}

By Proposition 1 and Examples 1 -- 4, isometric embeddings of  $J(l,m)$ in $\Gamma_{k}(V)$
exist for all pairs $l,m$ satisfying
$$\min\{m,l-m\}\le \min\{k,n-k\}$$
if the division ring $R$ is infinite.

\begin{theorem}
Let $n,k,l,m$ be natural numbers satisfying $n,l\ge 4$, $1<k<n-1$, $1<m<l-1$ and
$$m':=\min\{m,l-m\}\le \min\{k,n-k\}.$$
If ${\mathcal J}$ is the image of an isometric embedding of $J(l,m)$ in $\Gamma_{k}(V)$
then one of the following possibilities is realized:
\begin{enumerate}
\item[$\bullet$] there exists $M\in {\mathcal G}_{k-m'}(V)$ such that ${\mathcal J}$ is defined by
a $(2m')$-independent subset of $[M\rangle_{k-m'+1}$ consisting of $l$ elements {\rm (}Example \ref{exmp4}{\rm)},
\item[$\bullet$] there exists $N\in {\mathcal G}_{k+m'}(V)$ such that ${\mathcal J}$ is defined by
a $(2m')$-independent subset of $\langle N]_{k+m'-1}$ consisting of $l$ elements {\rm (}Example \ref{exmp5}{\rm)}.
\end{enumerate}
The image of every isometric embedding of $J(2m,m)$ in $\Gamma_{k}(V)$
is an apartment in a parabolic subspace of type
$$[M,N]_{k},\;\;\;M\in {\mathcal G}_{k-m}(V),\;\;N\in {\mathcal G}_{k+m}(V).$$
\end{theorem}

We assume that ${\mathcal G}_{0}(V)=\{0\}$ and ${\mathcal G}_{n}=\{V\}$.

\begin{cor}
Let ${\mathcal J}$ be the image of an isometric embedding of $J(n,k)$ in $\Gamma_{k}(V)$, $1<k<n-1$.
Then the following assertions are fulfilled:
\begin{enumerate}
\item[{\rm (1)}] if $2k<n$ then ${\mathcal J}$ is defined by a $(2k)$-independent subset of $\Pi_{V}$
consisting of $n$ elements {\rm (}Example \ref{exmp2}{\rm)}
or there exists $N\in {\mathcal G}_{2k}(V)$ such that
${\mathcal J}$ is defined by a $(2k)$-independent subset of $\langle N]_{2k-1}$ consisting of $n$ elements
{\rm (}Example \ref{exmp5}{\rm)},
\item[{\rm (2)}] if $2k>n$ then
${\mathcal J}$ is defined by a $(2n-2k)$-independent subset of $\Pi^{*}_{V}$ consisting of $n$ elements
{\rm (}Example \ref{exmp3}{\rm)}
or there exists $M\in {\mathcal G}_{2k-n}(V)$ such that ${\mathcal J}$ is defined by
a $(2n-2k)$-independent subset of $[M\rangle_{2k-n+1}$ consisting of $n$ elements {\rm (}Example \ref{exmp4}{\rm)},
\item[{\rm (3)}]
if $n=2k$ then ${\mathcal J}$ is an apartment of ${\mathcal G}_{k}(V)$.
\end{enumerate}
Therefore, the image of every isometric embedding of $J(n,k)$ in $\Gamma_{k}(V)$
is an apartment of ${\mathcal G}_{k}(V)$ if and only if $n=2k$.
\end{cor}

\section{Proof of Theorem 4}

Let $W$ be an $l$-dimensional vector space
and let $B$ be a base of $W$.
For each number $i\in \{1,\dots,l-1\}$
we denote by ${\mathcal A}_{i}$ the associated apartment of ${\mathcal G}_{i}(W)$.
Let $f:{\mathcal A}_{m}\to {\mathcal G}_{k}(V)$ be  an isometric embedding
of $\Gamma({\mathcal A}_{m})$ in $\Gamma_{k}(V)$.
Since $J(l,m)$ and $J(l,l-m)$ are isomorphic, we can assume that $m\le l-m$
(in other words, $m'=m$).
Then the diameter of $\Gamma({\mathcal A}_{m})$ is equal to $m$
and we have
$$m\le \min\{k,n-k\}.$$
The image of this embedding will be denoted by ${\mathcal J}$.
Then $f$ is an isomorphism of $\Gamma({\mathcal A}_{m})$ to $\Gamma({\mathcal J})$;
moreover, for any $S,U\in {\mathcal A}_{m}$ we have
$$d_{m}(S,U)=d_{k}(f(S),f(U)).$$
%We need to show that ${\mathcal J}$ is an apartment in a parabolic subspace of ${\mathfrak G}_{k}(V)$
%if $l=2m$.
The proof of Theorem 4 will be given in several steps.

{\bf 1}. Our first step is the following lemma.
\begin{lemma}\label{lemma1}
Every maximal clique of $\Gamma({\mathcal J})$
is contained in precisely one maximal clique of $\Gamma_{k}(V)$.
\end{lemma}

\begin{proof}
Let ${\mathcal Y}$ be a maximal clique of $\Gamma({\mathcal J})$.
Suppose that it is contained in at least two distinct maximal cliques of $\Gamma_{k}(V)$.
The intersection of two distinct maximal cliques of
$\Gamma_{k}(V)$ is the empty set, a point, or a line.
Thus ${\mathcal Y}$ is contained in a line $[M,N]_{k}$.
Clearly, $f^{-1}({\mathcal Y})$ is a maximal clique of $\Gamma({\mathcal A}_{m})$
and there is a maximal clique ${\mathcal Z}\ne f^{-1}({\mathcal Y})$ of $\Gamma({\mathcal A}_{m})$
intersecting $f^{-1}({\mathcal Y})$  in two vertices.
Then $f({\mathcal Z})$ is a maximal clique of $\Gamma({\mathcal J})$ intersecting
${\mathcal Y}$ in two vertices.
Every maximal clique of $\Gamma_{k}(V)$ contains a line or intersects this line
at most in a point.
Therefore, every maximal clique of $\Gamma_{k}(V)$ containing $f({\mathcal Z})$
contains the line $[M,N]_{k}$; hence it coincides with the star $[M\rangle_{k}$
or the top $\langle N]_{k}$.
The latter means that all vertices of $f({\mathcal Z})$
are adjacent with all vertices of ${\mathcal Y}$ which is impossible.
\end{proof}

A maximal clique ${\mathcal Y}$ of $\Gamma({\mathcal J})$ is said to be a {\it star} of ${\mathcal J}$
or a {\it top} of ${\mathcal J}$ if the maximal clique of $\Gamma_{k}(V)$
containing ${\mathcal Y}$ is a star or a top, respectively.

If the intersection of two distinct maximal cliques of $\Gamma({\mathcal A}_{m})$
consists of two vertices then these maximal cliques are of different types (one of them is a star and the other is a top).
The same holds for maximal cliques of $\Gamma({\mathcal J})$.
For any distinct maximal cliques ${\mathcal S}$ and ${\mathcal S}'$
of $\Gamma({\mathcal A}_{m})$
there is a sequence of maximal cliques of $\Gamma({\mathcal A}_{m})$
$${\mathcal S}={\mathcal S}_{0},{\mathcal S}_{1},\dots,{\mathcal S}_{i}={\mathcal S}'$$
such that $|S_{j-1}\cap S_{j}|=2$ for every $j\in \{1,\dots,i\}$.
This implies that for the mapping $f$ one of the following possibilities is realized:
\begin{enumerate}
\item[(A)] stars go to stars and tops go to tops,
\item[(B)] stars go to tops  and tops go to stars.
\end{enumerate}
In the next two steps we assume that the mapping $f$ satisfies (A).

{\bf 2}.
Denote by ${\mathcal J}_{k-1}$ the set of all $(k-1)$-dimensional subspaces of $V$
corresponding to the stars of ${\mathcal J}$.
The mapping $f$ induces a bijection
$$f_{m-1}:{\mathcal A}_{m-1}\to {\mathcal J}_{k-1}$$
satisfying
$$f({\mathcal A}_{m}\cap[ S\rangle_{m})={\mathcal J}\cap[ f_{m-1}(S)\rangle_{k}$$
for all $S\in {\mathcal A}_{m-1}$.
Then
$$f_{m-1}({\mathcal A}_{m-1}\cap \langle U]_{m-1})={\mathcal J}_{k-1}\cap \langle f(U)]_{k-1}$$
for all $U\in {\mathcal A}_{m}$.
The latter implies that $f_{m-1}$
sends adjacent elements of ${\mathcal A}_{m-1}$
to adjacent elements of ${\mathcal J}_{k-1}$.

Now we show that for all $S,U\in {\mathcal A}_{m-1}$
$$d_{m-1}(S,U)=d_{k-1}(f_{m-1}(S),f_{m-1}(U)).$$
Since $f_{m-1}$
sends adjacent elements of ${\mathcal A}_{m-1}$
to adjacent elements of ${\mathcal J}_{k-1}$,
$$d_{m-1}(S,U)\ge d_{k-1}(f_{m-1}(S),f_{m-1}(U)).$$
The condition $2m\le l$ guarantees the existence of $S',U'\in {\mathcal A}_{m}$
such that $S\subset S'$, $U\subset U'$ and
$$S\cap U=S'\cap U'.$$
Then
\begin{equation}\label{eq1}
d_{m-1}(S,U)=d_{m}(S',U')-1
\end{equation}
(indeed $d_{m-1}(S,U)=m-1-\dim(S\cap U)=m-1-\dim(S'\cap U')=d_{m}(S',U')-1$).
The mapping $f$ is an isometric embedding and
\begin{equation}\label{eq2}
d_{m}(S',U')=d_{k}(f(S'),f(U')).
\end{equation}
Since $f_{m-1}$ is induced by $f$,
we have
$$f_{m-1}(S)\subset f(S')\;\mbox{ and }\;f_{m-1}(U)\subset f(U')$$
which implies that
\begin{equation}\label{eq3}
\dim(f_{m-1}(S)\cap f_{m-1}(U))\le \dim(f(S')\cap f(U')).
\end{equation}
By \eqref{eq1} -- \eqref{eq3},
$$d_{m-1}(S,U)=d_{m}(S',U')-1=d_{k}(f(S'),f(U'))-1=k-1-\dim(f(S')\cap f(U'))\le$$
$$k-1-\dim(f_{m-1}(S)\cap f_{m-1}(U))=d_{k-1}(f_{m-1}(S),f_{m-1}(U)).$$
Thus
$$d_{m-1}(S,U)\le d_{k-1}(f_{m-1}(S),f_{m-1}(U))$$
and we get the required equality.

{\bf 3}.
So, $f_{m-1}$ is an isometric embedding of $\Gamma({\mathcal A}_{m-1})$ in $\Gamma_{k-1}(V)$.
Step by step, we construct a sequence of isometric embeddings
$$f_{i}:{\mathcal A}_{i}\to {\mathcal G}_{k-m+i}(V),\;\;\;\;i=m,\dots,1,$$
of $\Gamma({\mathcal A}_{i})$ in $\Gamma_{k-m+i}(V)$ such that $f_{m}=f$.
Denote by ${\mathcal J}_{k-m+i}$ the image of $f_{i}$ for each $i$.
If $i>1$ then
\begin{equation}\label{eq4}
f_{i}({\mathcal A}_{i}\cap[ S\rangle_{i})={\mathcal J}_{k-m+i}\cap[ f_{i-1}(S)\rangle_{k-m+i}
\;\;\;\;\;\forall\; S\in {\mathcal A}_{i-1}
\end{equation}
and
\begin{equation}\label{eq5}
f_{i-1}({\mathcal A}_{i-1}\cap \langle U]_{i-1})={\mathcal J}_{k-m+i-1}\cap \langle f(U)]_{k-m+i-1}
\;\;\;\;\;\forall\;U\in {\mathcal A}_{i}.
\end{equation}
This means that every element of ${\mathcal J}_{j}$, $j>k-m+1$
is the sum of two adjacent elements from ${\mathcal J}_{j-1}$.
Thus every element of ${\mathcal J}={\mathcal J}_{k}$
is the sum of some elements from ${\mathcal J}_{k-m+1}$.

Denote by $P_{1},\dots, P_{l}$ the elements of ${\mathcal A}_1$.
Then
$$T_{1}:=f_{1}(P_{1}),\dots,T_{l}:=f_{1}(P_{l})$$
form ${\mathcal J}_{k-m+1}$.
Using \eqref{eq4} and \eqref{eq5} we establish that
\begin{equation}\label{eq6}
f({\mathcal A}_{m}\cap[ P_{j}\rangle_{m})={\mathcal J}\cap[T_{j}\rangle_{k}.
\end{equation}

\begin{lemma}\label{lemma2}
A subspace of $V$ belongs to ${\mathcal J}$
if and only if it is
the sum of $m$ elements from ${\mathcal J}_{k-m+1}$.
\end{lemma}

\begin{proof}
Suppose that $i_{1},\dots,i_{m}\in \{1,\dots, l\}$ are distinct.
Then
$$S:=P_{i_{1}}+\dots +P_{i_{m}}\in {\mathcal A}_{m}$$
belongs to ${\mathcal A}_{m}\cap[ P_{j}\rangle_{m}$
if and only if $j\in \{i_{1},\dots,i_{m}\}$. Thus, by \eqref{eq6},
$T_{j}$ is contained in $f(S)$
only in the case when $j\in \{i_{1},\dots,i_{m}\}$.
It was noted above that $f(S)$ is the sum of some elements from ${\mathcal J}_{k-m+1}$;
hence
$$f(P_{i_{1}}+\dots +P_{i_{m}})=T_{i_{1}}+\dots +T_{i_{m}}$$
and we get the claim. \end{proof}

\begin{lemma}\label{lemma3}
The sum of any $2m$ elements from ${\mathcal J}_{k-m+1}$
is $(k+m)$-dimensional.
\end{lemma}

\begin{proof}
Suppose that $i_{1},\dots,i_{2m}\in\{1,\dots, l\}$ are distinct
and consider
$$S:=T_{i_{1}}+\dots+T_{i_{m}}\;\mbox{ and }\;U:=T_{i_{m+1}}+\dots+T_{i_{2m}}.$$
The intersection of the subspaces
$$f^{-1}(S)=P_{i_{1}}+\dots+P_{i_{m}}\;\mbox{ and }\;f^{-1}(U):=P_{i_{m+1}}+\dots+P_{i_{2m}}$$
is zero. Hence
$d_{m}(f^{-1}(S),f^{-1}(U))=m$
and, by our hypothesis,
$d_{k}(S,U)=m$.
The latter means that
$S+U$ is $(k+m)$-dimensional.
\end{proof}

The set ${\mathcal J}_{k-m+1}$ is formed by $l$
mutually adjacent elements of ${\mathcal G}_{k-m+1}(V)$.
It is not contained in a top (if ${\mathcal J}_{k-m+1}$ is a subset of a top
then the sum of all elements from ${\mathcal J}_{k-m+1}$ is $(k-m+2)$-dimensional
which contradicts Lemma \ref{lemma3}).
Thus
$${\mathcal J}_{k-m+1}\subset [M\rangle_{k-m+1},\;\;M\in {\mathcal G}_{k-m}(V).$$
Let $N$ be the sum of all elements from ${\mathcal J}_{k-m+1}$.
Then
\begin{equation}\label{eq7}
k+m\le \dim N\le k-m+l
\end{equation}
(this follows from Lemma \ref{lemma3} and the fact that $|{\mathcal J}_{k-m+1}|=l$).
It is clear that $M\subset N$
and ${\mathcal J}$
is contained in $[M,N]_{k}$.

By Lemmas \ref{lemma2} and $\ref{lemma3}$, ${\mathcal X}={\mathcal J}_{k-m+1}$
is a $(2m)$-independent subset of $[M\rangle_{k-m+1}$ and ${\mathcal J}$ coincides
with ${\mathcal J}_{m}({\mathcal X})$, see Example \ref{exmp4}.

If $l=2m$ then, by \eqref{eq7}, the subspace $N$ is $(k+m)$-dimensional
and ${\mathcal J}_{k-m+1}$ is a base of the projective space $[M,N]_{k-m+1}$.
This implies that ${\mathcal J}$ is an apartment of the parabolic subspace $[M,N]_{k}$.

{\bf 4}.
Suppose that $f$ satisfies (B) and consider the mapping $g:{\mathcal A}_{m}\to {\mathcal J}^{0}$
which transfers every $S\in{\mathcal A}_{m}$ to the annihilator of $f(S)$.
This is an isometric embedding of $J(l,m)$ in $\Gamma_{n-k}(V^{*})$
satisfying (A), see Fact 1.

As in the previous step, we establish the existence of $M\in {\mathcal G}_{n-k-m}(V^{*})$
and a $(2m)$-independent subset ${\mathcal X}\subset [M\rangle_{n-k-m+1}$
such that ${\mathcal J}^{0}$ coincides with ${\mathcal J}_{m}({\mathcal X})$.
Then $M^{0}\in {\mathcal G}_{k+m}(V)$ and ${\mathcal X}^{0}$ is a $(2m)$-independent
subset of $\langle M^{0}]_{k+m-1}$.
It is clear that ${\mathcal J}$ coincides with ${\mathcal J}^{*}_{k}({\mathcal X}^{0})$,
see Example \ref{exmp5}.

If $l=2m$ then ${\mathcal J}^{0}$ is an apartment of the parabolic subspace
$[M,N]_{n-k}$, where $N\in {\mathcal G}_{n-k+m}(V^{*})$.
Thus ${\mathcal J}$ is an apartment of the parabolic subspace
$[N^{0}, M^{0}]_{k}$ and $N^{0}\in {\mathcal G}_{k-m}(V)$.

\section{Permutations on finite subsets of projective spaces induced by semi-linear automorphisms}

Every permutation on an independent subset of $\Pi_{V}$ can be extended to an element of ${\rm PGL}(V)$,
such an extension is not unique even if our subset is a base of $\Pi_{V}$.
Indeed, we take a base $\langle x_{1}\rangle,\dots, \langle x_{n}\rangle$ of $\Pi_{V}$ and
consider any linear automorphism $u:V\to V$ which transfers every $x_{i}$ to $a_{i}x_{i}$;
if at least two $a_{i}$ are distinct then $u$ induces a non-identity element of ${\rm PGL}(V)$ whose restriction
to the base is identity.

\begin{prop}
Every permutation on an $n$-simplex of $\Pi_{V}$
can be extended to an element of ${\rm PGL}(V)$;
it can be uniquely extended
if and only if the division ring $R$ is commutative.
\end{prop}

\begin{proof}
See Propositions 1 and 2 in Section III.3 \cite{Baer}.
\end{proof}

By Proposition 2, every permutation on an $m$-simplex, $m<n$ can be extended to an element of ${\rm PGL}(V)$;
such an extension is not unique, since our simplex spans a proper subspace of $\Pi_{V}$.

\begin{prop}
Let ${\mathcal X}\subset {\mathcal G}_{1}(V)$ be a finite subset
such that every permutation on ${\mathcal X}$ is induced by a semi-linear automorphism of $V$.
Then ${\mathcal X}$ is an independent subset  or a simplex.
\end{prop}

\begin{cor}
Let ${\mathcal X}\subset {\mathcal G}_{1}(V)$ be a finite subset.
In the case when $R$ is commutative, the following conditions are equivalent:
\begin{enumerate}
\item[$\bullet$] every permutation on ${\mathcal X}$ can be uniquely extended to an element of ${\rm PGL}(V)$,
\item[$\bullet$] ${\mathcal X}$ is an $n$-simplex.
\end{enumerate}
\end{cor}

\begin{proof}
Let $P_{1},\dots, P_{n'}$ be the elements of ${\mathcal X}$.
For each $i\in\{1,\dots,n'\}$ we take a non-zero vector $x_{i}\in P_{i}$.
Suppose  that $\{P_{1},\dots,P_{m}\}$ is a maximal independent subset of ${\mathcal X}$.
If $m<n'$ then  the subset ${\mathcal X}$ is not independent.
In this case,
each $x_{j}$, $j>m$ is a linear combination of $x_{1},\dots,x_{m}$.
If this linear combination contains $x_{p}$ and does not contain $x_{q}$
for some $p,q\le m$ then
every semi-linear automorphism of $V$ inducing the transposition $(P_{p},P_{q})$
does not leave fixed $P_{j}$ which is impossible.
Therefore, for every $j>m$ we have
$$x_{j}=a_{j1}x_{1}+\dots+a_{jm}x_{m},$$
where each $a_{ji}$ is non-zero.
Then $P_{1},\dots,P_{m+1}$ form an $m$-simplex and we can assume that
$$x_{m+1}=x_{1}+\dots+x_{m};$$
in particular, ${\mathcal X}$ is an $m$-simplex if $n'=m+1$.
So, we need to show that the inequality $n'\ge m+2$ is impossible.

Suppose that $n'\ge m+2$ and
$$x_{m+2}=a_{1}x_{1}+\dots+ a_{m}x_{m}.$$
Consider any semi-linear automorphism $u:V\to V$ which induces the transposition $(P_{m+1},P_{m+2})$.
Then $u(x_{m+1})=ax_{m+2}$ for a non-zero scalar $a\in R$
and $s:=a^{-1}u$ is a semi-linear  automorphism of $V$
transferring $x_{m+1}$ to $x_{m+2}$.
Since $s(P_{i})=P_{i}$ if $i\le m$, we have
$$s(x_{i})=a_{i}x_{i},\;\;i\le m\;\mbox{ and }\;s(x_{m+2})=a^{2}_{1}x_{1}+\dots+a^{2}_{m}x_{m}\in P_{m+1}.$$
This implies the existence of a non-zero scalar $b\in R$ such that
$$a^{2}_{1}=\dots=a^{2}_{m}=b.$$
Thus for any $i,j\le m$ we get $a_{i}=\pm a_{j}$.
Since $x_{m+1}$ and $x_{m+2}$ are linearly independent, the equality $a_{i}=a_{j}$
does not hold for all pairs $i,j$.
Therefore, up to a permutation on $\{1,\dots, m\}$, we have
\begin{equation}\label{eq9}
x_{m+2}=a'(x_{1}+\dots+x_{p}-x_{p+1}-\dots-x_{m})
\end{equation}
with $1\le p<m$ and non-zero $a'\in R$.

Now suppose that $u$ is a semi-linear automorphism of $V$ which induces
the transposition $(P_{1},P_{m+1})$.
Then
$$u(x_{1})=c(x_{1}+\dots+x_{m})$$
and
$$u(x_{1}+\dots+x_{m})=c(x_{1}+\dots+x_{m})+c_{2}x_{2}+\dots+c_{m}x_{m}\in
P_{1}.$$
This implies that $c_{i}=-c$ and $u(x_{i})=-cx_{i}$ for all $i\in \{2,\dots,m\}$.
By \eqref{eq9},
$$u(x_{m+2})=
a''u(x_{1}+\dots+x_{p}-x_{p+1}-\dots-x_{m})=$$
$$a''(c(x_{1}+\dots+x_{m})-cx_{2}-\dots-cx_{p}+cx_{p+1}+\dots+cx_{m})=$$
$$a''c(x_{1}+2(x_{p+1}+\dots+x_{m}))$$
does not belong to $P_{m+2}$,
a contradiction.
\end{proof}

Every bijective transformation $f$ of a subset ${\mathcal X}\subset {\mathcal G}_{k}(V)$
defines a bijective transformation of ${\mathcal X}^{0}$ which maps every $S\in {\mathcal X}^{0}$
to $f(S^{0})^{0}$. The following statement is trivial.

\begin{lemma}\label{lemma4}
A bijective transformation of ${\mathcal X}\subset {\mathcal G}_{k}(V)$
is induced by a semi-linear automorphism $u:V\to V$ if and only if
the corresponding transformation of ${\mathcal X}^{0}$ is induced by the contragradient of $u$.
\end{lemma}

This lemma implies that the direct analogs of Propositions 2 and 3 hold for the dual projective space $\Pi^{*}_{V}$.

\section{Rigid embeddings}
A $J$-subset ${\mathcal J}\subset {\mathcal G}_{k}(V)$
is called {\it rigid} if
every automorphism of $\Gamma({\mathcal J})$ can be extended to an automorphism of $\Gamma_{k}(V)$,
in other words, ${\mathcal J}$ is the image of a rigid isometric embedding of a Johnson graph in $\Gamma_{k}(V)$.
In this section all rigid $J$-subsets will be classified.

The following observation is trivial.

\begin{fact}
Let $J(l,m)$ be the Johnson graph whose vertices are
the $m$-element subsets of an $l$-element set $I$.
If $l\ne 2m$ then every automorphism of $J(l,m)$
is induced by a permutation on the set $I$.
In the case when $l=2m$,
the group of all automorphisms of $J(l,m)$ is spanned by
the automorphisms induced by permutations on $I$
and the automorphism transferring each $m$-element subset of $I$ to its complement.
\end{fact}

Let us consider several examples.

\begin{exmp}{\rm
Let ${\mathcal A}$ be the apartment of ${\mathcal G}_{k}(V)$ defined by a base
$B=\{x_{i}\}^{n}_{i=1}$ of $V$.
If $n\ne 2k$ then every automorphism of $\Gamma({\mathcal A})$
is induced by a permutation on the set
$$\{P_{1}:=\langle x_{1}\rangle,\dots,P_{n}:=\langle x_{n}\rangle\};$$
thus it can be extended (not uniquely) to an element of ${\rm PGL}(V)$.
Suppose that $n=2k$.
In this case, the automorphism of $\Gamma({\mathcal A})$ transferring
every element of ${\mathcal A}$ to its complement
can not be extended to the automorphism of $\Gamma_{k}(V)$ induced by a semi-linear automorphism of $V$.
Assume that $R$ and $R^{*}$ are isomorphic; this guarantees that semi-linear isomorphisms of $V$ to $V^{*}$
exist.
Let $\{x^{*}_{i}\}^{n}_{i=1}$ be the base of $V^{*}$ dual to $B$; i.e.
$x^{*}_{i}x_{j}=\delta_{ij}$ ($\delta_{ij}$ is Kronecker symbol).
Consider any semi-linear isomorphism $s:V\to V^{*}$ which sends every $x_{i}$ to $x^{*}_{i}$
and denote by $g$ the associated automorphism of $\Gamma_{k}(V)$
(it maps every $S\in {\mathcal G}_{k}(V)$ to the annihilator of $s(S)$).
Then
$$g(P_{i_{1}}+\dots+P_{i_{k}})=P_{j_{1}}+\dots+P_{j_{k}},\;\;\;
\{j_{1},\dots,j_{k}\}=\{1,\dots,n\}\setminus \{i_{1},\dots,i_{k}\};$$
in other words, the restriction of $g$ to ${\mathcal A}$ is the automorphism
of $\Gamma({\mathcal A})$ sending every element of ${\mathcal A}$ to its complement.
Therefore, apartments of ${\mathcal G}_{k}(V)$ are rigid (under assumption that $R$ and $R^{*}$ are isomorphic if $n=2k$).
}\end{exmp}

\begin{exmp}{\rm
Let ${\mathcal A}$ be an apartment of a parabolic subspace
$$[M,N]_{k},\;\;M\in{\mathcal G}_{k-m}(V),\;\;N\in {\mathcal G}_{l+k-m}(V).$$
If $l\ne 2m$ then every automorphism of $\Gamma({\mathcal A})$
can be extended (not uniquely) to an element of ${\rm PGL}(V)$.
Suppose that $l=2m$.
If $n\ne 2k$ then the automorphism of $\Gamma({\mathcal A})$ transferring
every element of ${\mathcal A}$ to its complement can not be extended to an
automorphism of $\Gamma_{k}(V)$.
In the case when $n=2k$, the codimension of $N$ is equal to the dimension of $M$
and such an extension is possible if $R$ and $R^{*}$ are isomorphic.
}\end{exmp}

Now consider a $J$-subset ${\mathcal J}\subset {\mathcal G}_{k}(V)$
distinct from an apartment of ${\mathcal G}_{k}(V)$ and an apartment of a parabolic subspace of ${\mathfrak G}_{k}(V)$.
Then $\Gamma({\mathcal J})$ is isomorphic to $J(l,m)$ with $l\ne 2m$.
We assume that $m<l-m$.
By Theorem 4, one of the following possibilities is realized:
\begin{enumerate}
\item[$\bullet$]
${\mathcal J}={\mathcal J}_{m}({\mathcal X})$, where
${\mathcal X}$ is a finite $(2m)$-independent subset of $[M\rangle_{k-m+1}$
and $M\in {\mathcal G}_{k-m}(V)$,
\item[$\bullet$]
${\mathcal J}={\mathcal J}^{*}_{k}({\mathcal X})$, where
${\mathcal X}$ is a finite $(2m)$-independent subset of $\langle N]_{k+m-1}$
and $N\in {\mathcal G}_{k+m}(V)$.
\end{enumerate}
By our hypothesis, ${\mathcal X}$ is not independent.
Since $\Gamma({\mathcal J})$ is isomorphic to $J(l,m)$ and $l\ne 2m$,
every automorphism of $\Gamma({\mathcal J})$ is induced by a permutation on ${\mathcal X}$.
It follows from the results of the previous section that every automorphism of $\Gamma({\mathcal J})$
can be extended to an automorphism of $\Gamma_{k}(V)$ if and only if ${\mathcal X}$ is a simplex.
Also,
every automorphism of $\Gamma({\mathcal J})$ can be uniquely extended to an element of ${\rm PGL}(V)$
only in the case when $R$ is commutative and
${\mathcal X}$ is an $n$-simplex of $\Pi_{V}$ (then $2k\le n$) or
an $n$-simplex of $\Pi^{*}_{V}$ (then $2k\ge n$).

So, we get the following.

\begin{theorem}
Let $n,k,l,m$ and $m'$ be as in Theorem 4.
Let also ${\mathcal J}$ be the image of a rigid isometric embedding of $J(l,m)$ in $\Gamma_{k}(V)$.

If $l\ne 2m$ then one of the following possibilities is realized:
\begin{enumerate}
\item[$\bullet$]
$l\le n-k+m'+1$ and there exists $M\in {\mathcal G}_{k-m'}(V)$ such that ${\mathcal J}$ is
the set of all $(m'-1)$-faces of an $(l-1)$-simplex in the projective space $[M\rangle_{k-m'+1}$
{\rm (}Example \ref{exmp4}{\rm)},
\item[$\bullet$]
$l\le k+m'+1$ and there exists $N\in {\mathcal G}_{k+m'}(V)$ such that ${\mathcal J}$ is
the set of all $(m'-1)$-faces of an $(l-1)$-simplex in the projective space $\langle N]_{k+m'-1}$
{\rm (}Example \ref{exmp5}{\rm)},
\item[$\bullet$]
$l\le n-k+m'$ and there exist $M\in {\mathcal G}_{k-m'}(V)$ and $N\in {\mathcal G}_{l+k-m'}(V)$
such that ${\mathcal J}$ is an apartment of the parabolic subspace $[M,N]_{k}$,
\item[$\bullet$]
$l\le k+m'$ and there exist $M\in {\mathcal G}_{k+m'-l}(V)$ and $N\in {\mathcal G}_{k+m'}(V)$ such that ${\mathcal J}$ is
an apartment of the parabolic subspace $[M,N]_{k}$.
\end{enumerate}
Therefore, if $l>\max\{k,n-k\}+m'+1$ then there are no rigid isometric embeddings of $J(l,m)$ in $\Gamma_{k}(V)$.

If $l=2m$ then ${\mathcal J}$ is an apartment of a parabolic subspace
$$[M,N]_{k},\;\;M\in{\mathcal G}_{k-m}(V),\;\;N\in{\mathcal G}_{k+m}(V),$$
$n=2k$ and $R$ is isomorphic to $R^{*}$.

If every automorphism of the graph $\Gamma({\mathcal J})$ can be uniquely extended  to an element of ${\rm PGL}(V)$
then $R$ is commutative and one of the following possibilities is realized:
\begin{enumerate}
\item[$\bullet$]
$l=n+1$, $m'=k$, $2k\le n$ and ${\mathcal J}$ is the set of all $(k-1)$-faces of an $n$-simplex in $\Pi_{V}$
{\rm (}Example \ref{exmp2}{\rm)},
\item[$\bullet$]
$l=n+1$, $m'=n-k$, $2k\ge n$ and ${\mathcal J}$ is the set of all $(n-k-1)$-faces of an $n$-simplex in $\Pi^{*}_{V}$
{\rm (}Example \ref{exmp3}{\rm)}.
\end{enumerate} \end{theorem}

\begin{cor}
Let ${\mathcal J}$ be the image of a rigid isometric embedding of $J(n,k)$ in $\Gamma_{k}(V)$
distinct from an apartment of ${\mathcal G}_{k}(V)$.
If $2k<n$ then one of the following possibilities is realized:
\begin{enumerate}
\item[$\bullet$]
${\mathcal J}$ is the set of all $(k-1)$-faces of an $(n-1)$-simplex in $\Pi_{V}$ {\rm (}Example \ref{exmp2}{\rm)},
\item[$\bullet$]
$n=2k+1$ and there exists $N\in {\mathcal G}_{n-1}(V)$ such that
${\mathcal J}$ is the set of all $(k-1)$-faces of an $(n-1)$-simplex in the projective space $\langle N]_{n-2}$ {\rm (}Example \ref{exmp5}{\rm)}.
\end{enumerate}
In the case when $2k>n$, we have the following two possibilities:
\begin{enumerate}
\item[$\bullet$]
${\mathcal J}$ is the set of all $(n-k-1)$-faces of an $(n-1)$-simplex in $\Pi^{*}_{V}$ {\rm (}Example \ref{exmp3}{\rm)},
\item[$\bullet$]
$n=2k-1$ and there exists $M\in {\mathcal G}_{1}(V)$ such that ${\mathcal J}$ is
the set of all $(k-2)$-faces of an $(n-1)$-simplex in the projective space $[M\rangle_{2}$ {\rm (}Example \ref{exmp4}{\rm)}.
\end{enumerate}
\end{cor}

\subsection*{Acknowledgement}
I express my deep gratitude to the first referee for attracting my attention to rigid embeddings
and to the second referee for useful information concerning an example in \cite{CS}.

\end{document}